\documentclass[10pt]{article}
\usepackage{amsmath,amssymb,amsthm}
\usepackage{epsfig}
\usepackage{color}

\title{On the packing numbers in graphs}

\author {
Doost Ali Mojdeh and Babak Samadi\thanks{Corresponding author}\\
Department of Mathematics\\
University of Mazandaran, Babolsar, Iran\\
{\tt damojdeh@umz.ac.ir}\\
{\tt samadibabak62@gmail.com}\vspace{3mm}\\
Abdollah Khodkar\\
Department of Mathematics\\
University of West Georgia\\
Carrollton, GA 30118 USA\\
{\tt akhodkar@westga.edu}\vspace{3mm}\\
Hamid Reza Golmohammadi\\
Department of Mathematics\\
University of Tafresh, Tafresh, IRI\\
{\tt h.golmohamadi@tafreshu.ac.ir}\vspace{3mm}\\
}

\date{}

\newtheorem{theorem}{Theorem}[section]

\theoremstyle{definition}

\newtheorem{rem}[theorem]{Remark}

\begin{document}

\maketitle

\begin{abstract}
\noindent In this paper, we find upper bounds on the open packing and $k$-limited packing numbers with emphasis on the cases $k=1$ and $k=2$. We solve the problem of characterizing all connected graphs on $n$ vertices with $\rho_{o}(G)=n/\delta(G)$ which was presented in 2015 by Hamid and Saravanakumar. Also, by establishing a relation between the $k$-limited packing number and double domination number we improve two upper bounds given by Chellali and Haynes in 2005.
\vspace{1mm}\vspace{1mm}
\end{abstract}

\section{Introduction}

Throughout this paper, let $G$ be a finite graph with vertex set $V=V(G)$, edge set $E=E(G)$, minimum degree $\delta=\delta(G)$ and maximum degree $\Delta=\Delta(G)$. We use \cite{we} for terminology and notation which are not defined here. For any vertex $v \in V(G)$, $N(v)=\{u\in G\mid uv\in E(G)\}$ denotes the {\em open neighbourhood} of $v$ of $G$, and $N[v]=N(v)\cup \{v\}$ denotes its {\em closed neighbourhood}.\\
A subset $B\subseteq V(G)$ is a {\em packing} (an {\em open packing}) in $G$ if for every distinct vertices $u,v‎\in B‎$, $N[u]‎‎\cap N[v]‎=‎\emptyset‎$ ($N(u)‎‎\cap N(v)‎=‎\emptyset‎$). The {\em packing number} $\rho(G)$ ({\em open packing number} $\rho_{o}(G)$) is the maximum cardinality of a packing (an open packing) in $G$. These concepts have been studied in \cite{hs,mm}, and elsewhere.\\
 In \cite{hh}, Harary and Haynes introduced the concept of tuple domination numbers. Let $1\leq k\leq \delta(G)+1$. A set $D\subseteq V(G)$ is a {\em $k$-tuple dominating set} in $G$ if $|N[v]\cap D|\geq k$, for all $v\in V(G)$. The {\em $k$-tuple domination number}, denoted $\gamma_{\times k}(G)$, is the
smallest number of vertices in a $k$-tuple dominating set. In fact, the authors showed that every graph $G$ with $\delta \geq k-1$ has a $k$-tuple dominating set and hence a $k$-tuple domination number. When $k=2$, $‎\gamma‎_{‎\times‎2}‎‎(G)$ is called {\em double domination number} of $G$. For the special case $k=1$, $‎\gamma‎_{‎\times1}‎‎(G)=\gamma(G)$ is the well known domination number (see \cite{hhs}). The concept of tuple domination has been studied by several authors including \cite{gghr,msh}. In general, the reader can find a comprehensive information on various domination parameters in \cite{cfhv} and \cite{hhs}.\\
Gallant et al. \cite{gghr} introduced the concept of $k$-limited packing in graphs and exhibited some real-world applications of it to network security, market saturation and codes. A set of vertices $B \subseteq V$ is called a {\em $k$-limited packing set} in $G$ if $|N[v] \cap B| \leq k$ for all $v \in V$, where $k\geq 1$. The {\em $k$-limited packing number}, $L_{k}(G)$, is the largest number of vertices in a $k$-limited packing set. When $k=1$ we have $L‎_{1}(G)‎=‎\rho(G)‎$.\\
In this paper, we find upper bounds on the $k$-limited packing numbers. In Section $2$, we prove that $2(n-\ell+s\delta^{*})/(1+\delta^{*})$ is a sharp upper bound on $L_{2}(G)$ for a connected graph $G$ on $n\geq3$ vertices, where $\ell$, $s$ and $\delta^{*}=\delta^{*}(G)$ are the number of pendant vertices, the number of support vertices and $\min\{\deg(v)\mid v\ \mbox{is not a pendant vertex}\}$, respectively. Also, we give an upper bound on $L_{k}(G)$ (with characterization of all graphs attaining it) in terms of the order, size and $k$. In Section 3, we exhibit a solution to the problem of characterizing all connected graphs of order $n\geq2$ with $\rho_{o}(G)=n/\delta(G)$ posed in \cite{hsa}. Moreover, we prove that $\gamma‎_{‎\times‎2}‎‎(G)+‎\rho(G)‎‎\leq n-‎\delta(G)+2‎‎$ when $\delta(G)\geq2$. This improves two results in \cite{ch} given by Chellali and Haynes, simultaneously.


\section{Main results}

The $2$-limited packing number of $G$ has been bounded from above by $2n/(\delta(G)+1)$ (see \cite{msh}, as the special case $k=2$). We present the following upper bound which works better for all graphs with pendant vertices, especially trees. First, we recall that a support vertex is called a {\em weak support vertex} if it is adjacent to just one pendant vertex.

\begin{theorem}
Let $G$ be a connected graph of order $n\geq3$ with $s$ support vertices and $\ell$ pendant vertices. Then,
$$L_{2}(G)\leq \frac{2(n-\ell+s\delta^{*}(G))}{1+\delta^{*}(G)}$$
and this bound is sharp. Here $\delta^{*}(G)$ is the minimum degree taken over all vertices which are not pendant vertices.
\end{theorem}

\begin{proof}
Let $\{u_{1},\ldots,u_{s_{1}}\}$ be the set of weak support vertices in $G$. Let $G'$ be the graph of order $n'$ formed from $G$ by adding new vertices $v_{1},\ldots,v_{s_{1}}$ and edges $u_{1}v_{1},\ldots,u_{s_{1}}v_{s_{1}}$ to $G$ (we note that $G=G'$ if $G$ has no weak support vertex). Clearly
\begin{equation}\label{EQ13}
s'=s,n'=n+s_{1}\ \mbox{and}\ \ell'=\ell+s_{1}
\end{equation}
in which $s'$ and $\ell'$ are the number of support vertives and pendant vertices of $G'$, respectively. Moreover, since $n\geq3$ and $G$ is a connected graph, $G$ and $G'$ have the same set of vertices of degree at least two. Therefore,
\begin{equation}\label{EQU1}
\delta^{*}(G')=\delta^{*}(G)=\delta^{*}.
\end{equation}
Let $B'$ be a maximum $2$-limited packing in $G'$. Suppose to the contrary that there exists a support vertex $u$ in $G'$ for which $|N[u]\cap B'|\leq1$. Thus, there exists a pendant vertex $v\notin B'$ adjacent to $u$. It is easy to see that $B'\cup\{v\}$ is a $2$-limited packing in $G'$ which contradicts the maximality of $B'$. So, we may always assume that $B'$ contains two pendant vertices at each support vertex. This implies that all support vertices and the other $\ell_{u}-2$ pendant vertices for each support vertex $u$ belong to $V(G')\setminus B'$, in which $\ell_{u}$ is the number of pendant vertices adjacent to $u$.  Moreover, these pendant vertices have no neighbors in $B'$. Therefore,

\begin{equation}\label{EQ14}
|[B',V(G')\setminus B']|\leq2(n'-|B'|-\ell'+2s').
\end{equation}
On the other hand, each pendant vertex in $B'$ has exactly one neighbor in $V(G')\setminus B'$ and each of the other vertices in $V(G')\setminus B'$ has at least $\delta^{*}(G')-1$ neighbors in $B'$. Therefore,
\begin{equation}\label{EQ15}
(|B'|-2s')(\delta^{*}(G')-1)+2s'\leq|[B',V(G')\setminus B']|.
\end{equation}
Together inequalities (\ref{EQ14}) and (\ref{EQ15}) imply that
\begin{equation}\label{EQ16}
|B'|\leq \frac{2(n'-\ell'+s'\delta^{*}(G'))}{1+\delta^{*}(G')}.
\end{equation}
We now let $B$ be a maximum $2$-limited packing in $G$. Clearly, $B$ is a $2$-limited packing in $G'$, as well. Thus, $|B|\leq|B'|$. By (\ref{EQ13}),(\ref{EQU1}) and (\ref{EQ16}) we have
\begin{equation*}
L_{2}(G)=|B|\leq|B'|\leq \frac{2(n-\ell+s\delta^{*})}{1+\delta^{*}},
\end{equation*}
as desired.\\
To show that the upper bound is sharp, we consider the star $K_{1,n-1}$, for $n\geq3$, with $L_{2}(K_{1,n-1})=2$.
\end{proof}

It is easy to see that $L_{k}(G)=n$ if and only if $k\geq \Delta(G)+1$. So, in what follows we may always assume that $k\leq \Delta(G)$ when we deal with $L_{k}(G)$.\\
The following theorem provides an upper bound on $L‎_{k}(G)‎$ of a graph $G$ in terms of its order, size and $k$. Also, we bound $\rho_{o}(G)$ from above just in terms of the order and size.\\
First, we define $\Omega$ and $\Sigma$ to be the families of all graphs $G$ having the following properties, respectively.\vspace{1mm}\\
($p_{1}$)\ There exists a clique $S$ such that $G[V(G)\setminus S]$ is $(k-1)$-regular and every vertex in $S$ has exactly $k$ neighbors in $V(G)\setminus S$.\\
($p_{2}$)\ There exists a clique $S$ such that $G[V(G)\setminus S]$ is a disjoint union of copies of $K_{2}$ and every vertex in $S$ has exactly one neighbor in $V(G)\setminus S$.

\begin{theorem}
Let $G$ be a graph of order $n$ and size $m$. If $k\leq2(n-\sqrt{n^{2}-n-2m})‎‎$ or $\delta(G)\geq k-1$, then
$$L‎_{k}(G‎)‎\leq‎ n+‎k/2‎‎‎‎-‎\sqrt{‎k^2/4+(1-k)n+2m‎}$$
with equality if and only if $G\in \Omega$.\\
Furthermore, $\rho_{o}(G)\leq n-\sqrt{2m-n}$ for any graph $G$ with no isolated vertex. The bound holds with equality if and only if $G\in \Sigma$.
\end{theorem}

\begin{proof}
Let $L$ be a maximum $k$-limited packing set in $G$ and let $E(G[L])$ and $E(G[V\setminus L])$ be the edge set of subgraphs of $G$ induced by $L$ and $V\setminus L$, respectively. Clearly,
\begin{equation}\label{EQU0}
m=|E(G[L])|+|[L,V(G)\setminus L]|+|E(G[V\setminus L])|.
\end{equation}
Therefore,
\begin{equation}\label{EQ6}
2m\leq(k-1)|L|+2k(n-|L|)+(n-|L|)(n-|L|-1)‎.
\end{equation}
Solving the above inequality for $|L|$ we obtain
$$L_{k}(G)=|L|‎\leq‎\dfrac{2n+k-‎\sqrt{k^2+4(1-k)n+8m}‎}{2}‎‎,$$
as desired (note that $k\leq2(n-\sqrt{n^{2}-n-2m})‎‎$ or $\delta(G)\geq k-1$ implies that $‎k^2/4+(1-k)n+2m\geq0$).\\
We now suppose that the equality in the upper bound holds. Therefore $|E(G[L])|=(k-1)|L|$, $|[L,V(G)\setminus L]|=k(n-|L|)$ and $|E(G[V(G)\setminus L])|=(n-|L|)(n-|L|-1)$, by (\ref{EQ6}). This shows that $V(G)\setminus L$ is a clique satisfying the property ($p_{1}$). Thus, $G\in \Omega$. Conversely, suppose that $G\in \Omega$. Let $S$ be a clique of the minimum size among all cliques having the property ($p_{1}$). Then, it is easy to see that $L=V(G)\setminus S$ is a $k$-limited packing for which the upper bound holds with equality.\\
The proof of the second result is similar to the proof of the first one when $k=1$.
\end{proof}


\section{The special case $k=1$}

Hamid and Saravanakumar \cite{hsa} proved that
\begin{equation}\label{EQ17}
\rho_{o}(G)\leq\frac{n}{\delta(G)}
\end{equation}
for any connected graph $G$ of order $n\geq2$. Moreover, the authors
characterized all the regular graphs which attain the above bound. In general, they posed the problem of characterizing all connected graphs of order $n\geq2$ with equality in (\ref{EQ17}).
We solve this problem in this section. For this purpose, we define the family $\Gamma$ containing all graphs $G$ constructed as follows. Let $H$ be disjoint union of $t\geq1$ copies of $K_{2}$. Join every vertex $u$ of $H$ to $k$ new vertices as its private neighbors lying outside $V(H)$. Let $V=V(H)\cup(\cup_{u\in V(H)}pn(u))$, in which $pn(u)$ is the set of neighbors (private neighbors) of $u$ which lies outside $V(H)$. Add new edges among the vertices in $\cup_{u\in V(H)}pn(u)$ to construct a connected graph $G$ on the set of vertices in $V=V(G)$ with $\deg(v)\geq k+1$, for all $v\in \cup_{u\in V(H)}pn(u)$. Clearly, every vertex in $V(H)$ has the minimum degree $\delta(G)=k+1$ and every vertex in $\cup_{u\in V(H)}pn(u)$ has exactly one neighbor in $V(H)$.\\
We are now in a position to present the following theorem.
\begin{theorem}\label{THEO1}
Let $G$ be a connected graph of order $n\geq2$. Then, $\rho_{o}(G)=\frac{n}{\delta(G)}$ if and only if $G\in \Gamma$.
\end{theorem}

\begin{proof}
We first state a proof for (\ref{EQ17}).
Let $B$ be a maximum open packing in $G$. Every vertex in $V(G)$ has at most one neighbor in $B$ and hence every vertex in $B$ has at least $\delta(G)-1$ neighbors in $V(G)\setminus B$, by the definition of an open packing.
Thus,
\begin{equation}\label{EQ18}
(\delta(G)-1)|B|\leq|[B,V(G)\setminus B]|\leq n-|B|.
\end{equation}
Therefore, $\rho_{o}(G)=|B|\leq\frac{n}{\delta(G)}$.\\
Considering (\ref{EQ18}), we can see that the equality in (\ref{EQ17}) holds if and only $(\delta(G)-1)|B|=|[B,V(G)\setminus B]|$ and $|[B,V(G)\setminus B]|=n-|B|$. Since, $B$ is an open packing, this is equvalent to the fact that $H=G[B]$ is a disjoint union of $t=|B|/2$ copies of $K_{2}$, in which every vertex has the minimum degree and is adjacent to $k=\delta(G)-1$ vertices in $V(G)\setminus B$ and each vertex in $V(G)\setminus B$ has exactly one neighbor in $B$. Now, it is easy to see that the equality in (\ref{EQ17}) holds if and only $G\in \Gamma$.
\end{proof}

\begin{rem}
Similar to the proof of Theorem \ref{THEO1} we have $\rho(G)\leq n/(\delta(G)+1)$, for each connected graph $G$ of order $n$. Furthermore, the characterization of graphs $G$ attaining this bound can be obtained in a similar fashion by making some changes in $\Gamma$. It is sufficient to consider $H$ as a subgraph of $G$ with no edges in which every vertex has exactly $\delta(G)$ private neighbors lying outside $V(H)$.
\end{rem}

In \cite{ch}, Chellali and Haynes proved that for any graph $G$ of order $n$ with $‎\delta(G)‎\geq2‎‎$,
$$\gamma‎_{‎\times‎2}‎‎(G)+‎\rho(G)‎‎\leq n.$$
Also, they proved that
$$\gamma‎_{‎\times‎2}‎‎(G)\leq n-\delta(G)+1$$
for any graph $G$ with no isolated vertices.\\
We note that the second upper bound is trivial for $\delta(G)=1$. So, we may assume that $\delta(G)\geq2$. In the following theorem, using the concepts of double domination and $k$-limited packing, we improve these two upper bounds, simultaneously.

\begin{theorem}\label{Th.gamma2.rho}
Let $G$ be a graph of order $n$. If $\delta(G)\geq2$, then
$$\gamma‎_{‎\times‎2}‎‎(G)+‎\rho(G)‎‎\leq n-‎\delta(G)+2.‎‎$$
Furthermore, this bound is sharp.
\end{theorem}

\begin{proof}
Let $B$ be a maximum $‎(\delta(G)-1)‎$-limited packing set in $G$.
Every vertex in $B$ has at most $‎\delta(G)-2‎$ neighbours in $B$.
Therefore it has at least two neighbours in $V(G)\setminus B$.
On the other hand, every vertex in $V(G)\setminus B$ has at most $‎\delta(G)-1‎$
neighbours in $B$, hence it has at least one neighbour in $V(G)\setminus B$.
This implies that $V(G)\setminus B$ is a double dominating set in $G$.
Therefore,
\begin{equation}\label{EQ1}
‎\gamma‎_{‎\times‎2}‎‎(G)+L‎_{‎\delta(G)‎-1}(G)‎‎\leq n.‎
\end{equation}
Now let $1\leq k\leq\Delta(G)$ and let $B$ be a maximum $k$-limited
packing set in $G$. Then $|N[v]‎\cap B‎|‎\leq k‎$, for all $v‎\in V(G)‎$.
We claim that $B‎\neq V(G)‎$. If $B=V(G)$ and $u‎\in V(G)‎$ such that $\deg(u)=‎\Delta(G)‎$,
then $‎\Delta(G)+1=|N[u]‎\cap B‎|‎\leq k‎\leq ‎\Delta(G)‎‎‎‎$, a contradiction.
Now let $u‎\in V(G)\setminus B‎$. It is easy to check that $|N[v]‎\cap (B‎\cup \{u\}‎)‎|‎\leq k+1‎$,
for all $v‎\in V(G)‎$. Therefore $B‎\cup \{u\}‎$ is a $(k+1)$-limited packing set in $G$. Hence
$$L‎_{k+1}(G)‎‎\geq |B‎\cup \{u\}‎|‎=|B|+1=L‎_{k}(G)‎+1,$$
for $k=1, \ldots ,‎\Delta(G)$‎. Applying this inequality repeatedly leads to
$$L‎_{‎\delta‎-1}(G)‎\geq L‎_{1}‎‎‎(G)+‎\delta(G)-2=‎\rho(G)‎‎+‎\delta(G)-2‎.$$
Hence, $‎\gamma‎_{‎\times‎2}‎‎(G)+‎\rho(G)‎‎\leq n-‎\delta(G)+2‎‎$ by (\ref{EQ1}).
Finally, the upper  bound is sharp for the complete graph $K‎_{n}‎$ with $n\geq3$
\end{proof}


\end{document}